\documentclass[a4paper]{amsart}
\usepackage[margin=2.5cm]{geometry}
\usepackage{graphicx}
\usepackage{ulem}
\usepackage{bm}
\usepackage{color}
\usepackage{amssymb, amsmath, amsthm}
\usepackage{bm}
\theoremstyle{plain}
\usepackage{graphicx}
\usepackage[aboveskip=-5pt,position=bottom]{caption}
\usepackage{natbib}
\usepackage{url}
\usepackage[plainpages=false,pdfpagelabels,breaklinks,hyperfootnotes=false]{hyperref}

\newtheorem{lemma}{Lemma}[section]

\newtheorem{theorem}[lemma]{Theorem}
\newtheorem{cor}[lemma]{Corollary}
\newtheorem{prop}[lemma]{Proposition}
\newtheorem{exam}[lemma]{\normalfont \scshape
 Example}

\newcommand{\R}{\mathbb{R}}
\newcommand{\N}{\mathbb{N}}
\newcommand{\B}{\mathbb{B}}
\newcommand{\norm}[1]{\left\Vert#1\right\Vert}
\newcommand{\abs}[1]{\left\vert#1\right\vert}
\newcommand{\set}[1]{\left\{#1\right\}}

\newcommand{\bfV}{\bm{V}}

\newcommand{\bfX}{\bm{X}}

\newcommand{\bfZ}{\bm{Z}}

\newcommand{\bfeta}{\bm{\eta}}

\newcommand{\barE}{\bar E^-[0,1]}
\newcommand{\barC}{\bar C^-[0,1]}

\allowdisplaybreaks[4]
\begin{document}

\title[Local Asymptotic Normality]{Local Asymptotic Normality in $\delta$-Neighborhoods of Standard Generalized Pareto Processes}%
\author{Stefan Aulbach and Michael Falk}
\address{University of W\"{u}rzburg \\
Institute of Mathematics\\
Emil-Fischer-Str.~30\\
97074~W\"{u}rzburg, Germany} \email{stefan.aulbach@uni-wuerzburg.de\\
falk@mathematik.uni-wuerzburg.de}

\thanks{The first author was supported by DFG Grant FA 262/4-1.}

\begin{abstract}
\Citet{dehap06} provided models for spatial extremes in the case of stationarity, which
depend on just one parameter $\beta>0$ measuring tail dependence, and they proposed
different estimators for this parameter. This framework was supplemented in \citet{falk11}
by establishing  local asymptotic normality (LAN) of a corresponding point process of
exceedances above a high multivariate threshold, yielding in particular asymptotic efficient
estimators.

The estimators investigated in these papers are based on a finite set of points
$t_1,\dots,t_d$, at which observations are taken. We generalize this approach in the
context of functional extreme value theory (EVT). This more general framework allows
estimation over some spatial parameter space, i.\,e., the finite set of points $t_1,\dots,t_d$
is replaced by $t\in [a,b]$. In particular, we derive efficient estimators of $\beta$ based on
those processes in a sample of iid processes in $C[0,1]$ which exceed a given threshold
function.
\end{abstract}

\maketitle

\section{Introduction}

Suppose that the stochastic process $\bfV=(V_t)_{t\in[0,1]}\in
C[0,1]$ is a \textit{standard generalized Pareto process} (GPP) (\citet{buihz08}),
i.e., there exists $x_0>0$ such that
\[
P(\bfV\le f)=1+\log(G(f)),\qquad f\in\barE,\,\norm f_\infty\le x_0,
\]
where $\barE$ is the set of those bounded functions on $[0,1]$ that attain only nonpositive
values and which have a finite set of discontinuities. By $G$ we denote the functional
distribution function (df) of a \textit{standard extreme value process} (EVP)
$\bfeta=(\eta_t)_{t\in[0,1]}\in C[0,1]$, i.e.,
\[
G(f)=P(\bfeta\le f),\qquad f\in \barE,
\]
$P(\eta_t\le x)=\exp(x)$, $x\le 0$, $t\in[0,1]$,  and $\bfeta$ is \textit{max-stable}:
\[
P\left(\bfeta\le \frac fn\right)^n=P(\bfeta\le f),\qquad f\in\barE,\,n\in\N.
\]
All operations on functions such as $\le$, multiplication with a
constant etc.  are meant componentwise. For random functions,
i.e., stochastic processes such as $\bfV$, $\bfeta$ we use bold
letters, to distinguish these from nonrandom functions such as
$f$.

\Citet{dehap06} provided models for spatial extremes in the case of stationarity, which
depend on just one parameter $\beta>0$ measuring tail dependence, and they proposed
different estimators for this parameter. This framework was supplemented in \citet{falk11}
by establishing  local asymptotic normality (LAN) of a corresponding point process of
exceedances above a high multivariate threshold.

Precisely, it is assumed that for any $x_1,\dots,x_d\le 0$, $d\in\N$,
\begin{align} \label{eqn:finite_distribution_of_EVP}
P(\eta_{t_j}\le x_j,\,1\le j\le d)
&=\exp\left(-\int_{-\infty}^\infty\max_{j\le
d}\abs{x_j}\psi_\beta(s-t_j)\,ds\right),
\end{align}
where $\psi_\beta(s)=\beta\psi(\beta s)$ with a scale parameter $\beta>0$, and $\psi$ is a
continuous probability density  on $\R$ with $\psi(s)=\psi(-s)>0$ and $\psi(s)$, $s\ge 0$,
decreasing.

In the papers by \citet{dehap06} and \citet{falk11} the density $\psi$ is known and the
parameter $\beta$ is estimated. The estimators investigated in these papers are based on a
finite set of points $t_1<\dots<t_d$; estimation over some interval $t\in[a,b]$ seems to be
an open problem. This is the content of the present paper, which is organized as follows. In
Section~\ref{sec:auxiliary_results_and_tools} we compile some auxiliary results and tools, in
particular from functional extreme value theory (EVT). In
Section~\ref{sec:estimation_of_beta} we introduce our estimator of $\beta$ and establish
its asymptotic normality under the condition that the underlying observations
$\bfV^{(1)},\dots,\bfV^{(n)}$ are independent copies of a standard GPP $\bfV$. Local
asymptotic normality (LAN) of a corresponding point process of exceedances above a high
constant threshold function is established in Section~\ref{sec:lan_of_exceedances}. This is
achieved under the condition that the underlying observations are in a
$\delta$-neighborhood of a standard GPP. As an application we obtain from LAN-theory that
our estimator of $\beta$ is asymptotically efficient in this setup. For an account of
functional EVT we refer to \citet{dehaf06}; for a supplement including in particular basics of
GPP we refer to \citet{aulfaho11}.

\section{Auxiliary Results and Tools} \label{sec:auxiliary_results_and_tools}

In this section we compile several auxiliary results and tools. We start with the functional df
of a standard EVP $\bfeta\in C[0,1]$, whose finite dimensional marginal distributions (fidis)
are given by equation \eqref{eqn:finite_distribution_of_EVP}.

\begin{lemma} \label{lem:representation_of_df_of_EVP}
We have for any $f\in\barE$
\[
P(\bfeta\le f)=\exp\left(-\int_{-\infty}^\infty \sup_{t\in[0,1]}\left( \abs{f(t)}\psi(s-\beta t)\right)\,ds\right).
\]
\end{lemma}

\begin{proof}
The assertion follows from the fact that a probability measure is continuous from above
together with the dominated convergence theorem; note that $\int_{-\infty}^\infty
\sup_{t\in[0,1]}\psi(s-\beta t)\,ds<\infty$. Let $Q=\set{q_1,q_2,\dots}$ be a denumerable
and dense subset of $[0,1]$, which contains also the set of discontinuities of $f$. Recall
that $\bfeta\in C[0,1]$. From representation \eqref{eqn:finite_distribution_of_EVP} we
obtain
\begin{align*}
P(\bfeta\le f)&=P\left(\bigcap_{n\in\N} \set{\eta_{q_i}\le f(q_i),\,1\le i\le n}\right)\\
&=\lim_{n\to\infty} P(\eta_{q_i}\le f(q_i),\,1\le i\le n)\\
&=\lim_{n\to\infty} \exp\left(-\int_{-\infty}^\infty\max_{1\le i\le
n}\left(\abs{f(q_i)}\psi_\beta(s-q_i)\right)\,ds\right)\\
&=\exp\left(-\int_{-\infty}^\infty\lim_{n\to\infty}\left( \max_{1\le i\le
n}\left(\abs{f(q_i)}\psi_\beta(s-q_i)\right)\right)\,ds\right)\\
&=\exp\left(-\int_{-\infty}^\infty \sup_{t\in[0,1]}\left( \abs{f(t)}\psi(s-\beta t)\right)\,ds\right).
\end{align*}
\end{proof}

The preceding result provides the functional df $P(\bfV\le f)=1+\log(G(f))$  of the GPP
$\bfV$ in its upper tail.

\begin{cor} \label{cor:expansion_of_df_of_GPP}
There exists $x_0>0$ such that for the GPP $\bfV$  corresponding to  the EVP $\bfeta$ and
for any $f\in\barE$ with $\norm f_\infty\le x_0$
\begin{itemize}
\item[(i)]
\[
P(\bfV\le f)= 1-\int_{-\infty}^\infty \sup_{t\in[0,1]}\left( \abs{f(t)}\psi(s-\beta t)\right)\, ds,
\]
\item[(ii)]
\[
P(\bfV> f)= \int_{-\infty}^\infty \inf_{t\in[0,1]}\left( \abs{f(t)}\psi(s-\beta t)\right)\, ds.
\]
\end{itemize}
\end{cor}

\begin{proof}
While part (i) is an immediate consequence of of Lemma
\ref{lem:representation_of_df_of_EVP}, part (ii) follows from the inclusion-exclusion formula
as in the proof of Lemma 3.1 in \citet{falk11}.
\end{proof}

Note that
\[
\norm f_D:= \int_{-\infty}^\infty \sup_{t\in[0,1]}\left( \abs{f(t)}\psi(s-\beta t)\right)\, ds,
\]
defines a norm on the set $E[0,1]$, called \textit{$D$-norm}. By $E[0,1]$ we denote the
set of those functions on $[0,1]$, which are bounded and have a finite number of
discontinuities. The representation of a multivariate extreme value distribution (EVD) or of a
multivariate generalized Pareto distribution (GPD) in terms of a $D$-norm is well-known, see
\citet{fahure10}. This concept was extended  to functional spaces in \citet{aulfaho11}.
The fidis of the stochastic processes $\bfeta$ or $\bfV$ are obtained by considering the
function $f(t)=\sum_{i=1}^d x_i 1_{\set{t_i}}(t)\in\barE$, $x_i\le 0$, $t_i\in[0,1]$,
$d\in\N$. This norm satisfies, for example, the general inequality
\[
\norm f_\infty\le \norm f_D\le \norm f_\infty \norm 1_D,\qquad f\in E[0,1],
\]
where $1$ denotes the constant function one and $\norm
f_\infty:=\sup_{t\in[0,1]}\abs{f(t)}$  is the usual sup-norm. This inequality implies in
particular that each $D$-norm is equivalent with the sup-norm which, in turn, implies that
the $L_p$-norm $\norm f_p=\left(\int_0^1 \abs{f(t)}^p\,dt\right)^{1/p}$, with
$p\in[1,\infty)$, is not a $D$-norm.

The following auxiliary result is a crucial tool for the derivation of estimators of $\beta$.

\begin{lemma} \label{lem:computation_of_psi_integral}
We have
\[
\int_{-\infty}^\infty \inf_{t\in[0,1]}\psi(s-\beta t)\,ds=2\left(1-\Psi\left(\frac\beta 2\right)\right)=2\Psi\left(-\frac\beta2\right),
\]
where $\Psi(x)=\int_{-\infty}^x\psi(s)\,ds$.
\end{lemma}

\begin{proof}
We have
\begin{align*}
\int_{-\infty}^\infty \inf_{t\in[0,1]}\psi(s-\beta t)\,ds &= \int_{-\infty}^\infty \inf_{t\in[0,1]}\psi(\abs{s-\beta t})\,ds\\
&= \int_{-\infty}^\infty \min\left(\psi(\abs{s}), \psi(\abs{s-\beta})\right)\,ds\\
&=\int_{\beta/2}^\infty\psi(s)\,ds + \int_{-\infty}^{\beta/2}\psi(s-\beta)\,ds\\
&=2\Psi(-\beta/2).
\end{align*}
\end{proof}

\section{Estimation of $\beta$} \label{sec:estimation_of_beta}

A natural estimator of $\Psi(-\beta/2)$, based on independent copies
$\bfV^{(1)},\dots,\bfV^{(n)}$ of $\bfV$,  is by Corollary \ref{cor:expansion_of_df_of_GPP}
and Lemma \ref{lem:computation_of_psi_integral} given by
\[
\widehat\Psi_{c,n}:=\frac 1{2\abs c n}\sum_{i=1}^n 1_{(c,0]}(\bfV^{(i)}).
\]
Note the twofold meaning of $c$: In the denominator $2\abs c n$ this is just the absolute
value of the constant $c<0$, whereas in the term $1_{(c,0]}(\bfV^{(i)})$ we mean the
constant \textit{function} $c$, and we have $1_{(c,0]}(\bfV^{(i)})=1$ if and only if each
component satisfies $V^{(i)}_t>c$, $t\in[0,1]$. There should be no risk of confusion.

The law of large number implies
\[
\widehat\Psi_{c,n}\to_{n\to\infty} \Psi\left(-\frac \beta 2\right)\quad \mbox{a.s.}
\]
and, thus,
\[
\widehat\beta_{c,n}:=-2\Psi^{-1}\left(\widehat\Psi_{c,n}\right) \to_{n\to\infty}\beta \quad \mbox{a.s.},
\]
where $F^{-1}(q):=\inf\set{t\in\R:\,F(t)\ge q}$, $q\in(0,1)$, denotes the generalized
inverse of a df $F$.

The Moivre-Laplace theorem implies asymptotic normality of $\widehat\Psi_{c,n}$ and
$\widehat\beta_{c,n}$, i.e., the next result is a functional counterpart of Proposition 3.3 in
\citet{falk11}.

\begin{prop}
For $c<0$ close enough to 0 we have
\begin{align*}
&n^{1/2}\left(\hat\Psi_{c,n}-\Psi\left(-\frac \beta 2\right)\right)\\
&\to_D
N\left(0,\frac{\Psi\left(-\frac \beta 2\right)\left(1-2\abs{c}\Psi
\left(-\frac \beta 2\right)\right)}{2\abs{c}}\right)
\end{align*}
and
\begin{align*}
n^{1/2}\left(\hat\beta_{c,n}-\beta\right)
\to_D
N\left(0,\frac{2\Psi\left(-\frac \beta 2\right)\left(1-2\abs{c}\Psi
\left(-\frac\beta 2\right)\right)}{\abs{c}\psi^2
\left(-\frac\beta 2\right)}\right).
\end{align*}
\end{prop}

We now consider a stochastic process $\bfX\in \bar C^{-}[0,1]:=\{f\in C[0,1]:f\le 0\}$,
whose upper tail is in a $\delta$-neighborhood of that of a GPP $\bfV\in C^{-}[0,1]$ with
$D$-norm $\norm f_D=\int_{-\infty}^\infty \sup_{t\in[0,1]}(\abs{f(t)}\psi(s-\beta t))\,ds$.
Precisely, we require that
\begin{equation} \label{cond:delta_neighborhood_of_GPP}
P(\bfX>cf)=P(\bfV> cf)\left(1+c^\delta K(f)+r(c,f)\right) \tag{C}
\end{equation}
for $c\in(0,1)$ and $f\in\barE$ with $\norm f_\infty\le\varepsilon_0$ for some
$\varepsilon_0>0$,  where $K:\barE\to\R$ is a function and the remainder $r(c,f)$ is of
order $o\left(c^\delta\right)$ as $c\to 0$. The next result is an immediate consequence of
Corollary \ref{cor:expansion_of_df_of_GPP}.

\begin{lemma} \label{lem:survival_probability_of_X}
Suppose that the stochastic process $\bfX\in \bar C^{-}[0,1]$
satisfies condition \eqref{cond:delta_neighborhood_of_GPP}.
Then we obtain for $c\in(0,1)$ and $f\in\barE$ with $\norm
f_\infty\le\varepsilon_0$
\begin{align*}
P(\bfX>cf)&= c\left(\int_{-\infty}^\infty \inf_{t\in[0,1]}(\abs{f(t)}\psi(s-\beta t))\,ds\right) \left(1+c^\delta K(f)+r(c,f)\right).
\end{align*}
\end{lemma}

In what follows we show how a process $\bfX$ satisfying condition
\eqref{cond:delta_neighborhood_of_GPP} can be generated. From \citet{aulfaho11} we
conclude that there is a stochastic process $\bfZ=(Z_t)_{t\in[0,1]}$ on $[0,1]$ with
continuous sample paths and $0\le Z_t\le m$, $E(Z_t)=1$, $t\in[0,1]$, for some constant
$m\ge 1$, such that
\begin{align*} \label{eqn:representation_of_D-norm_via_generator_Z}
\norm f_D&=\int_{-\infty}^\infty \sup_{t\in[0,1]}(\abs{f(t)}\psi(s-\beta t))\,ds\nonumber\\
&=E\left(\sup_{t\in[0,1]}(\abs{f(t)}Z_t)\right),\qquad f\in E[0,1].
\end{align*}
The stochastic process $\bfZ$ is called \textit{generator} of the $D$-norm. Conversely,
each process $\bfZ$ with the above properties generates a $D$-norm via $\norm f_D:=
E\left(\sup_{t\in[0,1]}(\abs{f(t)}Z_t)\right)$, $f\in E[0,1]$. For every $D$-norm $\norm
\cdot_D$ there exists a standard EVP $\bfeta\in C[0,1]$ with functional df $P(\bfeta\le
f)=\exp(-\norm f_D)$, $f\in \barE$. While a generator $\bfZ$ is in general not uniquely
determined, the \textit{generator constant} $E\left(\sup_{t\in[0,1]}(Z_t)\right)$ $=\norm
1_D$ is. We refer to \citet{aulfaho11} for details.

Put
\begin{equation} \label{eqn:generation_of_gpp}
  \bfV:=(V_t)_{t\in[0,1]}:=\left(\max\left(-\frac U{Z_t},M\right)\right)_{t\in[0,1]},
\end{equation}
where $U$ and $\bfZ$ are independent, $U$ is a uniformly on $(0,1)$ distributed rv and
$M<0$ is an arbitrary constant. We incorporate the constant $M$ to ensure that
$V_t>-\infty$ for each $t\in[0,1]$, as $Z_t$ may attain the value zero. The continuous
process $\bfV$ is a GPP, as we have for $f\in\barE$ with $\norm f_\infty\le\min(|M|,1/m)$
\begin{align*}
P(\bfV\le f)&=P(U\ge \abs{f(t)} Z_t,\,0\le t\le 1)\\
&= P(U\ge \sup_{t\in[0,1]}(\abs{f(t)} Z_t))\\
&=1-E\left(\sup_{t\in[0,1]}(\abs{f(t)} Z_t)\right)\\
&=1-\norm f_D.
\end{align*}
We have, moreover,
\begin{align*}
P(\bfV> f)&=P(U< \abs{f(t)} Z_t,\,0\le t\le 1)\\
&= P(U\le \inf_{t\in[0,1]}(\abs{f(t)} Z_t))\\
&=E\left(\inf_{t\in[0,1]}(\abs{f(t)} Z_t)\right)\\
&=\int_{-\infty}^\infty \inf_{t\in[0,1]}(\abs{f(t)}\psi(s-\beta t))\,ds,
\end{align*}
where the final equality is a consequence of Corollary \ref{cor:expansion_of_df_of_GPP}, part (ii).

Replace now the rv $U$ in \eqref{eqn:generation_of_gpp}  by a rv $Y>0$, which is also
independent of $\bfZ$ and whose df $H$ is continuous and satisfies
\begin{equation} \label{eqn:distribution_function_of_Y_in_gpp}
H(u)=u+ A u^{1+\delta}+o(u^{1+\delta})\qquad\mbox{as }u\downarrow 0
\end{equation}
with some constant $A\in\R$. The standard exponential distribution, for instance, satisfies
this condition with $\delta=1$ and $A=-1/2$. The process
\begin{equation} \label{eqn:generation_of_process_in_neighborhood_of_gpp}
\bfX:=(X_t)_{t\in[0,1]}:=\left(\max\left(-\frac Y {Z_t},M\right)\right)_{t\in[0,1]}
\end{equation}
then satisfies condition \eqref{cond:delta_neighborhood_of_GPP} with
\[
K(f)=A \frac{E\left(\inf_{t\in[0,1]}(\abs{f(t)}Z_t)^{1+\delta}\right)} {E\left(\inf_{t\in[0,1]}(\abs{f(t)}Z_t)\right)},
\]
which has to be interpreted as zero if the denominator vanishes.

The following theorem is the main result of this section. We will see in
Section~\ref{sec:lan_of_exceedances} using LAN theory that it implies that $\hat
\Psi_{c_n,n}$ is an asymptotically efficient estimator sequence in an appropriate model.

\begin{theorem} \label{thm:asymptotic_normality_of_estimators_in_neighborhood_of_GPP}
Suppose that the stochastic process $\bfX\in \bar C^{-}[0,1]$ satisfies condition
\eqref{cond:delta_neighborhood_of_GPP}. If the sequence of thresholds $c_n<0$, $n\in\N$,
satisfies $c_n\to 0$, $n\abs{c_n}\to\infty$,
$n\abs{c_n}^{1+2\delta}\to\mathrm{const}\ge 0$ as $n\to\infty$, then we obtain
\begin{enumerate}
\item
\begin{equation*}
 (n\abs{c_n})^{1/2}\left(\hat\Psi_{c_n,n}-\Psi\left(-\frac\beta2
\right)\right)
 \to_D
N\left(\mathrm{const}^{1/2}\mu,\frac12\Psi\left(-\frac\beta 2
\right)\right),
\end{equation*}
\item
\begin{equation*}
 (n\abs{c_n})^{1/2}\left(\hat\beta_{c_n,n}-\beta\right)
 \to_D N\left(-\frac{2\,\mathrm{const}^{1/2}\mu}{
\psi\left(-\frac\beta2\right)},\frac{2\Psi\left(-
\frac\beta2\right)}{
\psi^2\left(-\frac\beta2\right)}\right),
\end{equation*}
\end{enumerate}
 where $\mu:=K(-1) \Psi(-\beta/2)$.
\end{theorem}

\begin{proof}
From Lemma \ref{lem:survival_probability_of_X} we obtain
\begin{align} \label{eqn:expansion_of_survival_probability_of_X}
&\frac 1{{\abs c}^\delta}\left(\frac {P(\bfX>c)}{\abs c}- \int_{-\infty}^\infty \inf_{t\in[0,1]}\psi(s-\beta t)\,ds\right)\nonumber\\
&\to_{c\uparrow 0} K(-1) \int_{-\infty}^\infty \inf_{t\in[0,1]}\psi(s-\beta t)\,ds\nonumber\\
&= 2K(-1)\Psi\left(-\frac\beta2\right).
\end{align}
Write
\begin{align*}
&(n\abs{c_n})^{1/2} \left(\widehat\Psi_{c_n,n} - \Psi\left(-\frac\beta 2\right)\right)\\
&= (n\abs{c_n})^{1/2} \left(\frac 1{2n\abs{c_n}}\sum_{j=1}^n \left(1_{(c_n,0]}(\bfX_j)-P(\bfX>c_n)\right)\right)\\
&\hspace*{2cm}+ (n\abs{c_n})^{1/2} \left(\frac{P(\bfX>c_n)}{2\abs{c_n}} - \Psi\left(-\frac\beta 2\right)\right)\\
&=:\eta_n+b_n.
\end{align*}
The Moivre-Laplace theorem implies
\[
\eta_n\to_DN\left(0,\frac 12 \Psi\left(-\frac\beta 2\right)\right),
\]
and expansion \eqref{eqn:expansion_of_survival_probability_of_X} yields
\begin{align*}
b_n&=\frac{\left(n\abs{c_n}^{1+2\delta}\right)^{1/2}} 2 \frac 1{\abs{c_n}^\delta} \left(\frac{P(\bfX>c_n)} {\abs{c_n}} - \int_{-\infty}^\infty \inf_{t\in[0,1]} \psi(s-\beta t)\,ds\right)\\
&\to_{n\to\infty} \mathrm{const}^{1/2}\mu.
\end{align*}
Equally, one concludes
\begin{align*}
&(n\abs{c_n})^{1/2}\left(\hat\beta_{c_n,n}-\beta\right) \\*
&=2 (n\abs{c_n})^{1/2}\left(\Psi^{-1}\left(\Psi\left(-\frac \beta 2\right)\right)- \Psi^{-1}\left(\widehat\Psi_{c_n,n}\right)\right)\\
&= 2 (n\abs{c_n})^{1/2} \left(\Psi^{-1}\right)'(\xi)\left(\Psi\left(-\frac \beta 2\right)- \widehat\Psi_{c_n,n}\right)\\
&\to_D N\left(-\frac {2\mathrm{const}^{1/2}\mu}{\psi\left(-\frac \beta 2\right)}, \frac{2\Psi\left(-\frac \beta 2\right)}{\psi^2\left(-\frac \beta 2\right)}\right)
\end{align*}
by Slutsky's lemma, with $\xi$ between $\widehat\Psi_{c_n,n}$
and $\Psi(-\beta/2)$. This completes the proof.
\end{proof}

The idea suggests itself to substitute the constant threshold $c$ by a suitable threshold
function $f\in \barE$ and to consider, with $c<0$,
\begin{align*}
\widehat\Psi_{f,c,n}&:=\frac 1{2\abs{c}n}\sum_{i=1}^n 1(\bfV^{(i)}>\abs c f)\\
&\to_{n\to \infty}\frac 1{2\abs c} P(\bfV>\abs c f)\\
&=\frac{1}{2}\int_{-\infty}^\infty\inf_{t\in[0,1]} \left(\abs{f(t)}\psi(s-\beta t)\right)\,ds
\end{align*}
almost surely by the law of large numbers and Corollary \ref{cor:expansion_of_df_of_GPP}.

The fact that with constant function $f=-1$, the above integral equals by
Lemma~\ref{lem:computation_of_psi_integral}
\[
\int_{-\infty}^\infty\inf_{t\in[0,1]} \psi(s-\beta t)\,ds=2\Psi\left(-\frac \beta 2\right)
\]
was the crucial observation for the derivation of an estimator of $\beta$. Substituting the
constant function $f=-1$ by an arbitrary function $f\in\barE$ can, however, lead to
surprising consequences, as the following example shows.

\begin{exam}\upshape
Take $\psi(s)=2^{-1}\exp(-\abs s)$, $s\in\R$, and $f(t):=-\exp(-t)$, $t\in[0,1]$. Then we
have for any $\beta\in[0,1]$
\[
T(f,\beta,\psi):= \int_{-\infty}^\infty\inf_{t\in[0,1]} \left(\abs{f(t)}\psi(s-\beta t)\right)\,ds=\exp(-1),
\]
i.e., the functional $T(f,\beta,\psi)$ is not capable to discriminate between different values
of $\beta\in[0,1]$. For $\beta> 1$ one obtains, however,
\begin{equation*}
  T(f, \beta, \psi) = \exp\left(-\frac{1+\beta}{2}\right).
\end{equation*}
\end{exam}

The question, whether for each underlying density $\psi$ there exists an \textit{optimal}
threshold function $f=f_\psi$, is an open problem.

\section{LAN of Exceedances} \label{sec:lan_of_exceedances}

Let $\bfX^{(i)}$, $1\le i\le n$, be independent copies of a stochastic process $\bfX\in
\barC$,  which satisfies condition \eqref{cond:delta_neighborhood_of_GPP}. Choose $c<0$.
In this section we establish local asymptotic normality (LAN) of the point process of
exceedances
\[
N_{n,c}(B):=\sum_{i=1}^n\varepsilon_{\sup_{t\in[0,1]}\left(X^{(i)}_t/c\right)}(B\cap [0,1)),\qquad B\in\B,
\]
where $\B$ denotes the $\sigma$-field of Borel sets in $\R$. Note that for $s\in(0,1]$
\[
\sup_{t\in[0,1]} \frac {X_t}c<s\iff \bfX>sc,
\]
i.e., the random point measure $N_{n,c}$ actually represents those processes  among
$\bfX^{(1)},\dots,\bfX^{(n)}$ which are exceedances above the constant function $c$.

It is by Theorem \ref{thm:asymptotic_normality_of_estimators_in_neighborhood_of_GPP},
part (i),  quite convenient to substitute the parameter $\beta>0$ in the family
$\psi_\beta(\cdot)=\beta\psi(\beta\cdot)$ by the parameter
\[
\vartheta:=2\Psi\left(-\frac \beta 2\right)\in(0,1).
\]
Fix $\vartheta_0\in(0,1)$. We require that the family of univariate df
$F_{\vartheta,c}(s):=P_\vartheta(\bfX>s c)$, $\vartheta\in (0,1)$, $s>0$, satisfies for
$s\in (0,1)$, $c_0\le c< 0$ for some $c_0<0$, and $\vartheta$ close to $\vartheta_0$ the
expansion
\begin{equation} \label{cond:differentiable_spectral_neighborhood}
\frac{f_{\vartheta,c}(s)} {f_{\vartheta_0,c}(s)} := \frac{\frac d{ds}P_\vartheta(\bfX>s c)}{\frac d{ds}P_{\vartheta_0}(\bfX>s c)} = 1+L(\vartheta-\vartheta_0)+ r_{\vartheta_0}(s,\vartheta,c), \tag{D}
\end{equation}
with
\[
r_{\vartheta_0}(s,\vartheta,c) = o(\abs{\vartheta-\vartheta_0})+O\left(\abs c^\gamma\right)
\]
uniformly for $s\in(0,1)$, $c_0\le c\le 0$ and $\vartheta$ close to $\vartheta_0$, where the
constants $L\in\R$ and $\gamma>0$ may depend  on $\vartheta_0$. Note that condition
\eqref{cond:differentiable_spectral_neighborhood} implies in particular that
$P_\vartheta\left(\sup_{t\in[0,1]}\left(X_t/c\right)=s\right)=0$ and, thus,
$F_{\vartheta,c}(s)= P_\vartheta\left(\sup_{t\in[0,1]}\left(X_t/c\right)\le s\right)$, $s>0$,
is actually a df on $[0,\infty)$.

Condition \eqref{cond:differentiable_spectral_neighborhood} is, for example, satisfied with
$L=1/\vartheta_0$ and $r_{\vartheta_0}=0$ if $\bfX$ is a GPP. We can also use the
approach from definition \eqref{eqn:generation_of_process_in_neighborhood_of_gpp} to
generate a process
\begin{equation} \label{eqn:generation_of_process_in_neighborhood_of_gpp_for_lan}
\bfX:=\left(\max\left(-\frac Y {Z_t}, M\right)\right)_{t\in[0,1]}.
\end{equation}
In addition to condition \eqref{eqn:distribution_function_of_Y_in_gpp} we require that the
continuous df $H$ of the rv $Y>0$ satisfies the expansion
\[
H(u)=u+A u^{1+\delta}+r(u),\qquad 0<u<1,
\]
with some constant $A\in\R$, where the function $r$ is differentiable on $(0,1)$ with
bounded derivative and $r'(u)=o(u^\delta)$ as $u\downarrow 0$. Then condition
\eqref{cond:differentiable_spectral_neighborhood} is satisfied with $L=1/\vartheta_0$.

Denote by $Y_1,\dots,Y_{\tau(n)}$ those rv among $\sup_{t\in[0,1]}X_t^{(i)}/c$, $1\le
i\le n$, with $\sup_{t\in[0,1]}X_t^{(i)}/c<1$, in the order of their outcome. Then we have
\[
N_{n,c}(B)=\sum_{k\le \tau(n)}\varepsilon_{Y_k}(B),\qquad
B\in\mathbb{B}.
\]
From Theorem 1.4.1 in \citet{reiss93} we obtain that
$Y_1,Y_2,\dots$ are independent copies of a rv $Y$ with df
\[
P_\vartheta(Y\le
t)=\frac{P_\vartheta(\bfX>tc)}{P_\vartheta(\bfX>c)},\qquad 0\le
t\le 1,
\]
under parameter $\vartheta>0$, and they are independent of the total number $\tau(n)$,
which is binomial $B(n,P_\vartheta(\bfX>c))$-distributed.

Since the distribution $\mathcal{L}_{\vartheta,c}(Y)$ of $Y$ under $\vartheta$ is by
condition \eqref{cond:differentiable_spectral_neighborhood} dominated by
$\mathcal{L}_{\vartheta_0,c}(Y)$ for $\vartheta$ in a neighborhood of $\vartheta_0$ and
$c_0\le c<0$, the distribution $\mathcal{L}_\vartheta(N_{n,c})$ of $N_{n,c}$  is
dominated by $\mathcal{L}_{\vartheta_0}(N_{n,c})$, see, e.g. Theorem 3.1.2 in
\citet{reiss93}. Precisely, $N_{n,c}$ is a random element in the set
$\mathbb{M}:=\{\mu=\sum_{1\le j\le n}\varepsilon_{y_j}:\;y_j\ge 0,\, j\le
n,\,n=0,1,2,\dots\}$ of finite point measures on $([0,\infty),\mathbb{B}\cap[0,\infty))$,
equipped with the smallest $\sigma$-field $\mathcal{M}$ such that for any
$B\in\mathbb{B}\cap[0,\infty)$ the projection $\pi_B:\,\mathbb{M}\to\set{0,1,2,\dots}$,
$\pi_B(\mu):=\mu(B)$, is measurable; we refer to Section~1.1 in \citet{reiss93} for
technical details.

From \citet[Example 3.1.2]{reiss93} we conclude that $\mathcal{L}_\vartheta(N_{n,c})$
has the $\mathcal{L}_{\vartheta_0}(N_{n,c})$-density
\begin{align*}
&\frac{d\mathcal{L}_\vartheta(N_{n,c})}{d\mathcal{L}_{\vartheta_0}(N_{n,c})}(\mu) \\*
&=\left(\prod_{i=1}^{\mu((0,1))}\frac{f_{\vartheta,c}(y_i)}
{f_{\vartheta_0,c}(y_i)}
\frac{P_{\vartheta_0}(\bfX>c)}{P_\vartheta(\bfX>c)}\right)\\
&\hspace*{2cm}\times
\left(\frac{P_\vartheta(\bfX>c)}{P_{\vartheta_0}(\bfX>c)}
\right)^{\mu((0,1))}
\left(\frac{1-P_\vartheta(\bfX>c)}{1-P_{\vartheta_0}(\bfX>c)}\right)^{n-\mu((0,1))}
\end{align*}
if $\mu=\sum_{i=1}^{\mu((0,1))}\varepsilon_{y_i}$ and $\mu((0,1))\le
n$. The loglikelihood ratio is, consequently,
\begin{align} \label{eqn:loglikelihood_ratio_of_point_processes}
&L_{n,c}(\vartheta\mid\vartheta_0)\nonumber\\*
&=\log\left\{\frac{d\mathcal{L}_\vartheta(N_{n,c})}
{d\mathcal{L}_{\vartheta_0}(N_{n,c})}(N_{n,c}) \right\}\nonumber\\
&=\sum_{k\le\tau(n)}
\log\left(\frac{f_{\vartheta,c}(Y_k)}{f_{\vartheta_0,c}(Y_k)}
\frac{P_{\vartheta_0}(\bfX>c)}{P_\vartheta(\bfX>c)} \right)\nonumber\\
&\hspace*{1cm}
+\tau(n)\log\left(\frac{P_\vartheta(\bfX>c)}{P_{\vartheta_0}(\bfX>c)}
\right) +
(n-\tau(n))\log\left(\frac{1-P_\vartheta(\bfX>c)}{1-P_{\vartheta_0}
(\bfX>c)} \right).
\end{align}

We let in the sequel $c=c_n$ depend on the sample size $n$ with $c_n\uparrow 0$ and,
equally, $\vartheta=\vartheta_n$ with $\vartheta_n\to\vartheta_0$ as $n\to\infty$.
Precisely, we put with arbitrary $\xi\in\R$
\[
\vartheta_n:=\vartheta_n(\xi):=\vartheta_0+\frac{\xi}{(n\abs{c_n})^{1/2}}.
\]
The following theorem is the main result of this section. It is analogous to Theorem 5.1 in
\citet{falk11}, whose proof carries over.

\begin{theorem} \label{thm:LAN} Suppose that $\psi(s)=\psi(-s)$ and that
$\psi(s)$, $s\ge 0$, is decreasing. Suppose, further,  that
$n\abs{c_n}\to_{n\to\infty}\infty$ and that
\begin{equation} \label{cond:gamma_condition_on_probability_and_x_n}
n\abs{c_n}^{1+2\min(\delta,\gamma)}\to_{n\to\infty}0.
\end{equation}
Then we obtain the expansion
\begin{align*}
L_{n,c_n}(\vartheta_n\mid\vartheta_0)&=\xi L Z_n-\frac{\xi^2
L^2\vartheta_0}2+o_{P_{\vartheta_0}}(1)\\
&\to_{D_{\vartheta_0}}N
\left(-\frac{\xi^2L^2\vartheta_0}2,\xi^2L^2\vartheta_0\right)
\end{align*}
with
\begin{equation} \label{eqn:asymptotic_normality_of_Z_n_under_theta_0}
Z_n:=\frac{\tau(n)-n\abs{c_n}\vartheta_0}
{(n\abs{c_n})^{1/2}}\to_{D_{\vartheta_0}}N(0,\vartheta_0).
\end{equation}
\end{theorem}

The above result reveals that the complete information about the underlying parameter that
is contained in the exceedances $Y_1,\dots,Y_{\tau(n)}$ is, actually, contained in their
number $\tau(n)$ as $n$ increases. This is in complete accordance with the results in
\citet{falk98}, where this phenomenon was studied for general truncated empirical
processes. The result here is, however, derived under more specialized conditions.

Theorem \ref{thm:LAN} together with the Haj\'{e}k-LeCam convolution theorem provides the
asymptotically minimum variance within the classes of regular estimators of $\vartheta_0$.
This class of estimators $\widetilde\vartheta_n$ is defined by the property that they are
asymptotically unbiased under
$\vartheta_n=\vartheta_n(\xi)=\vartheta_0+\xi(n\abs{c_n})^{-1/2}$ with
$\vartheta_0\in(0,1)$ for any $\xi\in\R$, precisely,
\[
(n\abs{c_n})^{1/2}\left(\widetilde\vartheta_n-\vartheta_n\right)
\to_{D_{\vartheta_n}}Q_{\vartheta_0},
\]
where the limit distribution $Q_{\vartheta_0}$ does not depend on $\xi$; see, e.g.
Sections~8.4 and 8.5 in \citet{pfanz94}.

By LeCam's first lemma (see, e.g., \citet[Chapter~3, Theorem~1]{lecy90}) we obtain that
under $\vartheta_n=\vartheta_n(\xi)$
\begin{align*}
L_{n,c_n}(\vartheta_n\mid\vartheta_0)&=\xi
LZ_n-\frac{\xi^2L^2\vartheta_0}2+o_{P_{\vartheta_n}}(1)\\
&\to_{D_{\vartheta_n}}N\left(\frac{\xi^2L^2\vartheta_0}2,
\xi^2L^2\vartheta_0\right)
\end{align*}
with
\begin{equation} \label{eqn:asymptotic_distribution_of_central_sequence_u_alternative}
Z_n\to_{D_{\vartheta_n}}N(\xi L\vartheta_0,\vartheta_0).
\end{equation}

An efficient estimator of $\vartheta_0$ within the class of regular estimators has necessarily
the minimum limiting variance
\[
\sigma^2_{\mathrm{minimum}}=\frac1{L^2\vartheta_0},
\]
which is the inverse of the limiting variance of the \textit{central sequence} $LZ_n$ under
$\vartheta_0$ (\citet[Theorem 8.4.1]{pfanz94}).

Consider the estimator
\[
\widehat\vartheta_n:=\frac{\tau(n)}{n\abs{c_n}}.
\]
Then we have with
$\vartheta_n=\vartheta_n(\xi)=\vartheta_0+\xi(n\abs{c_n})^{-1/2}$
\[
(n\abs{c_n})^{1/2}\left(\widehat\vartheta_n-\vartheta_n\right) =
(n\abs{c_n})^{1/2}\left(\frac{\tau(n)}{n\abs{c_n}}-\vartheta_0\right)-\xi=Z_n-\xi.
\]
The estimator $\widehat\vartheta_n$ is, consequently, not a regular estimator since we
have by \eqref{eqn:asymptotic_distribution_of_central_sequence_u_alternative}
\[
(n\abs{c_n})^{1/2}\left(\widehat\vartheta_n-\vartheta_n\right) =
Z_n-\xi
\to_{D_{\vartheta_n}}N\left(\xi(L\vartheta_0-1),\vartheta_0\right),
\]
where the limiting distribution depends on $\xi$ unless $L=1/\vartheta_0$. Its asymptotic
relative efficiency, defined as the ratio of the limiting variances under $\vartheta_0$ is
\[
ARE(\vartheta_0) = \frac{\vartheta_0}{\sigma^2_{\mathrm{minimum}}} =
L^2\vartheta_0^2.
\]
Recall that $L=1/\vartheta_0$ if $\bfX$ follows a GPP or if $\bfX$ is in a neighborhood of a
GPP as in \eqref{eqn:generation_of_process_in_neighborhood_of_gpp_for_lan} and, thus,
$\widehat\vartheta_n$ is in this case regular and asymptotically efficient.

\begin{cor}
Suppose in addition to the conditions of Theorem \ref{thm:LAN} that
 $\bfX$ is a GPP or it is in a neighborhood of a GPP as in \eqref{eqn:generation_of_process_in_neighborhood_of_gpp_for_lan}. Then
$\widehat\vartheta_n=\tau(n)/(n\abs{c_n})$, $n\in\N$, is a regular
estimator sequence with asymptotic minimum variance $\vartheta_0$ within the
class of regular estimators.
\end{cor}

A regular estimator sequence can in general be obtained as follows. Suppose that
$\vartheta_n^*$ is a solution of the equation
\[
P_{\vartheta_n^*}(\bfX>c_n)=\frac{\tau(n)}n.
\]
Since $\tau(n)$ is under $\vartheta_0$ binomial $B(n,P_{\vartheta_0}(\bfX>c_n))$
distributed, $\vartheta_n^*$ is, actually, the maximum likelihood estimator of
$\kappa_0=P_{\vartheta_0}(\bfX>c_n)$ for the family
$\set{B(n,\kappa)=B(n,P_{\vartheta}(\bfX>c_n)):\,\vartheta\in(0,1)}$. We suppose
consistency of the sequence $\vartheta_n^*$, $n\in\N$. Then we obtain from condition
\eqref{cond:differentiable_spectral_neighborhood} the expansion
\begin{align*}
\frac{\tau(n)}n&=P_{\vartheta_n^*}(\bfX>c_n)\\
&=\int_0^1\left(1+L(\vartheta_n^*-\vartheta_0)+r_{\vartheta_0}
(u,\vartheta_n^*,c_n) \right)f_{\vartheta_0,c_n}(u)\,du\\
&=\left(1+L(\vartheta_n^*-\vartheta_0)+o_{P_{\vartheta_0}}\left(\abs{
\vartheta_n^*-\vartheta_0}\right)+O\left(\abs{c_n}^\gamma\right)\right)
P_{\vartheta_0}(\bfX>c_n),
\end{align*}
which implies
\[
(n\abs{c_n})^{1/2}\left(\vartheta_n^*-\vartheta_0\right)=
\frac1{L\vartheta_0}Z_n+o_{P_{\vartheta_0}}(1).
\]

As a consequence we obtain from \eqref{eqn:asymptotic_normality_of_Z_n_under_theta_0}
and \eqref{eqn:asymptotic_distribution_of_central_sequence_u_alternative} with
$\vartheta_n=\vartheta_n(\xi)$
\[
(n\abs{c_n})^{1/2}\left(\widehat\vartheta_n^*-\vartheta_n\right)
\to_{D_{\vartheta_n}}N\left(0,\frac1{L^2\vartheta_0^2}\right),
\]
and, thus, $\vartheta_n^*$, $n\in\N$, is a regular estimator sequence with asymptotic
minimum variance.

\bibliographystyle{enbib_arXiv}
\bibliography{evt}

\begin{thebibliography}{10}
\providecommand{\natexlab}[1]{#1}
\providecommand{\url}[1]{\texttt{#1}}
\providecommand{\urlprefix}{URL }
\providecommand{\selectlanguage}[1]{\relax}
\providecommand{\bibinfo}[2]{#2}
\providecommand{\href}[2]{#2}
\providecommand{\eprint}[2][]{\href{#1}{#2}}

\bibitem[{Aulbach et~al.(2011)Aulbach, Falk, and Hofmann}]{aulfaho11}
\bibinfo{author}{\textsc{Aulbach, S.}}, \bibinfo{author}{\textsc{Falk, M.}},
  and \bibinfo{author}{\textsc{Hofmann, M.}} (\bibinfo{year}{2011}).
\newblock \bibinfo{title}{On extreme value processes and the functional
  $D$-norm}.
\newblock \bibinfo{type}{Tech. Rep.}, \bibinfo{institution}{University of
  W{\"u}rzburg}.
\newblock \bibinfo{note}{Submitted},
  \eprint[http://arxiv.org/abs/1107.5136]{{\ttfamily arXiv:1107.5136
  [math.PR]}}.

\bibitem[{Buishand et~al.(2008)Buishand, de~Haan, and Zhou}]{buihz08}
\bibinfo{author}{\textsc{Buishand, T.~A.}}, \bibinfo{author}{\textsc{de~Haan,
  L.}}, and \bibinfo{author}{\textsc{Zhou, C.}} (\bibinfo{year}{2008}).
\newblock \bibinfo{title}{On spatial extremes: With application to a rainfall
  problem}.
\newblock \textit{\bibinfo{journal}{Ann. Appl. Stat.}}
  \textbf{\bibinfo{volume}{2}}, \bibinfo{pages}{624--642}.

\bibitem[{Falk(1998)}]{falk98}
\bibinfo{author}{\textsc{Falk, M.}} (\bibinfo{year}{1998}).
\newblock \bibinfo{title}{Local asymptotic normality of truncated empirical
  processes}.
\newblock \textit{\bibinfo{journal}{Ann. Statist.}}
  \textbf{\bibinfo{volume}{26}}, \bibinfo{pages}{692--718}.

\bibitem[{Falk(2011)}]{falk11}
\bibinfo{author}{\textsc{Falk, M.}} (\bibinfo{year}{2011}).
\newblock \bibinfo{title}{Local asymptotic normality in a stationary model for
  spatial extremes}.
\newblock \textit{\bibinfo{journal}{J. Multivariate Anal.}}
  \textbf{\bibinfo{volume}{102}}, \bibinfo{pages}{48--60}.

\bibitem[{Falk et~al.(2010)Falk, H{\"u}sler, and Reiss}]{fahure10}
\bibinfo{author}{\textsc{Falk, M.}}, \bibinfo{author}{\textsc{H{\"u}sler, J.}},
  and \bibinfo{author}{\textsc{Reiss, R.-D.}} (\bibinfo{year}{2010}).
\newblock \textit{\bibinfo{title}{Laws of Small Numbers: Extremes and Rare
  Events}}.
\newblock \bibinfo{edition}{3rd} ed. \bibinfo{publisher}{Birkh{\"a}user},
  \bibinfo{address}{Basel}.

\bibitem[{de~Haan and Ferreira(2006)}]{dehaf06}
\bibinfo{author}{\textsc{de~Haan, L.}}, and \bibinfo{author}{\textsc{Ferreira,
  A.}} (\bibinfo{year}{2006}).
\newblock \textit{\bibinfo{title}{Extreme Value Theory: An Introduction}}.
\newblock Springer Series in Operations Research and Financial Engineering.
  \bibinfo{publisher}{Springer}, \bibinfo{address}{New York}.

\bibitem[{de~Haan and Pereira(2006)}]{dehap06}
\bibinfo{author}{\textsc{de~Haan, L.}}, and \bibinfo{author}{\textsc{Pereira,
  T.~T.}} (\bibinfo{year}{2006}).
\newblock \bibinfo{title}{Spatial extremes: Models for the stationary case}.
\newblock \textit{\bibinfo{journal}{Ann. Statist.}}
  \textbf{\bibinfo{volume}{34}}, \bibinfo{pages}{146--168}.

\bibitem[{LeCam and Yang(1990)}]{lecy90}
\bibinfo{author}{\textsc{LeCam, L.}}, and \bibinfo{author}{\textsc{Yang,
  G.~L.}} (\bibinfo{year}{1990}).
\newblock \textit{\bibinfo{title}{Asymptotics in Statistics: Some Basic
  Concepts}}.
\newblock Springer Series in Statistics. \bibinfo{publisher}{Springer},
  \bibinfo{address}{New York}.

\bibitem[{Pfanzagl(1994)}]{pfanz94}
\bibinfo{author}{\textsc{Pfanzagl, J.}} (\bibinfo{year}{1994}).
\newblock \textit{\bibinfo{title}{Parametric Statistical Theory}}.
\newblock \bibinfo{publisher}{De Gruyter}, \bibinfo{address}{Berlin}.

\bibitem[{Reiss(1993)}]{reiss93}
\bibinfo{author}{\textsc{Reiss, R.-D.}} (\bibinfo{year}{1993}).
\newblock \textit{\bibinfo{title}{A Course on Point Processes}}.
\newblock \bibinfo{publisher}{Springer}, \bibinfo{address}{New York}.

\end{thebibliography}

\end{document}